\documentclass[11pt,a4paper,reqno]{amsart}
\usepackage[english]{babel}
\usepackage[T1]{fontenc}
\usepackage{verbatim}
\usepackage{palatino}
\usepackage{amsmath}
\usepackage{mathabx}
\usepackage{amssymb}
\usepackage{amsthm}
\usepackage{amsfonts}
\usepackage{graphicx}
\usepackage{esint}
\usepackage{color}
\usepackage{mathtools}
\DeclarePairedDelimiter\ceil{\lceil}{\rceil}
\DeclarePairedDelimiter\floor{\lfloor}{\rfloor}
\usepackage[colorlinks = true, citecolor = black]{hyperref}
\author{Tuomas Orponen}
\title[On two theorems of Lutz and Stull]{Combinatorial proofs of \\ two theorems of Lutz and Stull}
\address{Department of Mathematics and Statistics\\ University of Jyv\"askyl\"a,
P.O. Box 35 (Mattilanniemi MaD)\\
FI-40014 University of Jyv\"askyl\"a\\
Finland}
\email{tuomas.t.orponen@jyu.fi}
\date{\today}
\thanks{T.O. is supported by the Academy of Finland via the projects \emph{Quantitative rectifiability in Euclidean and non-Euclidean spaces} and \emph{Incidences on Fractals}, grant Nos. 309365, 314172, 321896.}

\newcommand{\R}{\mathbb{R}}

\newcommand{\N}{\mathbb{N}}

\newcommand{\Z}{\mathbb{Z}}

\newcommand{\calD}{\mathcal{D}}
\newcommand{\calH}{\mathcal{H}}

\newcommand{\Hd}{\dim_{\mathrm{H}}}
\newcommand{\Pd}{\dim_{\mathrm{p}}}
\newcommand{\Bd}{\overline{\dim}_{\mathrm{B}}\,}

\newcommand{\spa}{\operatorname{span}}
\newcommand{\diam}{\operatorname{diam}}

\newcommand{\dist}{\operatorname{dist}}

\def\Barint_#1{\mathchoice
          {\mathop{\vrule width 6pt height 3 pt depth -2.5pt
                  \kern -8pt \intop}\nolimits_{#1}}%
          {\mathop{\vrule width 5pt height 3 pt depth -2.6pt
                  \kern -6pt \intop}\nolimits_{#1}}%
          {\mathop{\vrule width 5pt height 3 pt depth -2.6pt
                  \kern -6pt \intop}\nolimits_{#1}}%
          {\mathop{\vrule width 5pt height 3 pt depth -2.6pt
                  \kern -6pt \intop}\nolimits_{#1}}}

\numberwithin{equation}{section}

\theoremstyle{plain}
\newtheorem{thm}[equation]{Theorem}
\newtheorem*{"thm"}{"Theorem"}

\newtheorem{lemma}[equation]{Lemma}

\newtheorem{pigeon}{Pigeon}

\theoremstyle{definition}

\newtheorem{definition}[equation]{Definition}

\theoremstyle{remark}

\newcommand{\nref}[1]{(\hyperref[#1]{#1})}

\DeclareMathSymbol{\intop}  {\mathop}{mathx}{"B3}

\begin{document}

\begin{abstract} Recently, Lutz and Stull used methods from algorithmic information theory to prove two new Marstrand-type projection theorems, concerning subsets of Euclidean space which are not assumed to be Borel, or even analytic. One of the theorems states that if $K \subset \R^{n}$ is any set with equal Hausdorff and packing dimensions, then
\begin{displaymath} \Hd \pi_{e}(K) = \min\{\Hd K,1\} \end{displaymath} 
for almost every $e \in S^{n - 1}$. Here $\pi_{e}$ stands for orthogonal projection to $\spa(e)$.

The primary purpose of this paper is to present proofs for Lutz and Stull's projection theorems which do not refer to information theoretic concepts. Instead, they will rely on combinatorial-geometric arguments, such as discretised versions of Kaufman's "potential theoretic" method, the pigeonhole principle, and a lemma of Katz and Tao. A secondary purpose is to slightly generalise Lutz and Stull's theorems: the versions in this paper apply to orthogonal projections to $m$-planes in $\R^{n}$, for all $0 < m < n$. \end{abstract}

\maketitle

\tableofcontents

\section{Introduction}

This paper contains combinatorial-geometric proofs of two recent projection theorems of Lutz and Stull, namely \cite[Theorems 2 \& 3]{MR3854026}. The original arguments were based on algorithmic information theory, and the intriguing \emph{point-to-set principle}, established previously by Lutz and Lutz \cite{MR3811993}: this principle -- or rather a formula -- expresses the Hausdorff and packing dimensions of an arbitrary set $K \subset \R^{n}$ as the supremum over the (relativized) \emph{pointwise dimensions} of elements $x \in K$. The information theoretic approach is very novel, so it is perhaps natural to ask: can more "conventional" (from the fractal geometers' point of view!) arguments yield the same results? The purpose of this note is to show that they can, at least in the case of the two projection theorems in \cite{MR3854026}. 

I move to the details, and start with some notation: for $0 < m < n$, the notation $G(n,m)$ refers to the \emph{Grassmannian manifold} of $m$-dimensional subspaces of $\R^{n}$, and $\gamma_{n,m}$ is a natural Haar measure on $G(n,m)$, see \cite[\S 3.9]{zbMATH01249699} for more details. Write $\pi_{V} \colon \R^{n} \to V$ for the orthogonal projection to an $m$-plane $V \in G(n,m)$. Hausdorff and packing dimensions will be denoted $\Hd$ and $\Pd$, respectively (see Section \ref{s:prelim} for precise definitions).

For context, I recall the \emph{Marstrand-Mattila projection theorem}:
\begin{thm}[Marstrand-Mattila]\label{mm} Let $0 < m < n$, and let $K \subset \R^{n}$ be an analytic set. Then 
\begin{displaymath} \Hd \pi_{V}(K) = \min\{\Hd K,m\} \qquad \text{for $\gamma_{n,m}$ a.e. } V \in G(n,m). \end{displaymath}
\end{thm}
The case $(n,m) = (2,1)$ is due to Marstrand \cite{MR0063439}, and the general case is due to Mattila \cite{MR0409774}. The analyticity assumption cannot be dropped, at least if the reader believes in the continuum hypothesis: using the continuum hypothesis, Davies \cite[Theorem 1*]{MR567545} constructed a $1$-dimensional set $K \subset \R^{2}$ with zero-dimensional projections to all lines. It would be interesting to know if counterexamples can be constructed without the continuum hypothesis.

In \cite{MR3854026}, Lutz and Stull showed that the analyticity condition can be dropped, however, in somewhat weaker variants of Theorem \ref{mm}:

\begin{thm}\label{main2} Let $0 < m < n$, and let $K \subset \R^{n}$ be a set with $\Hd K = \Pd K$. Then
\begin{displaymath} \Hd \pi_{V}(K) = \min\{\Hd K,m\} \qquad \text{for $\gamma_{n,m}$ a.e. } V \in G(n,m). \end{displaymath}
\end{thm}

\begin{thm}\label{main1} Let $0 < m < n$, and let $K \subset \R^{n}$. Then 
\begin{displaymath} \Pd \pi_{V}(K) \geq \min\{\Hd K,m\} \qquad \text{for $\gamma_{n,m}$ a.e. } V \in G(n,m). \end{displaymath} 
\end{thm}

To be precise, only the cases $m = 1$ of Theorems \ref{main2}-\ref{main1} were established in \cite{MR3854026}, and I do not know if the case $m > 1$ would present additional difficulties for the information theoretic approach in \cite{MR3854026}. In the present paper, I will reprove Theorems \ref{main2}-\ref{main1} with combinatorial arguments, which are essentially the same for all $0 < m < n$. These arguments consist of "$\delta$-discretised" versions of the \emph{potential theoretic} method, due Kaufman \cite{MR248779}, and multiple applications of the pigeonhole principle. Some tools are also taken from Katz and Tao's paper \cite{MR1856956}. However, I will repeat details to the extent that this paper is essentially self-contained.

Finally, I mention that Theorem \ref{main1} cannot be improved to 
\begin{displaymath} \dim_{\mathrm{P}} \pi_{V}(K) \geq \min\{\dim_{\mathrm{P}} K,m\} \qquad \textrm{for $\gamma_{n,m}$ a.e. } V \in G(n,m),\end{displaymath}
even if $K \subset \R^{n}$ is compact. Examples of compact sets $K \subset \R^{2}$ with $\dim_{\mathrm{P}} \pi_{L}(K) < \dim_{\mathrm{P}} K \leq 1$ for all $L \in G(2,1)$ were constructed by M. J\"arvenp\"a\"a \cite{MR1316487}. The optimal lower bounds -- for analytic sets -- were established by Falconer and Howroyd \cite{MR1357045}. For a more recent approach, see \cite{2019arXiv190111014F}. As far as I know, the possibility of such lower bounds for arbitrary sets has not been investigated.

\subsection{Notation} An open ball in $\R^{n}$ with centre $x \in \R^{n}$ and radius $r > 0$ will be denoted $B(x,r)$. For $A,B > 0$, the notation $A \lesssim_{p_{1},\ldots,p_{k}} B$ means that there exists a constant $C \geq 1$, depending only on the parameters $p_{1},\ldots,p_{k}$, such that $A \leq CB$. The two-sided inequality $A \lesssim_{p} B \lesssim_{p} A$ is abbreviated to $A \sim_{p} B$, and $A \gtrsim_{p_{1},\ldots,p_{k}} B$ is synonymous to $B \lesssim_{p_{1},\ldots,p_{k}} A$. The notation "$\log$" means logarithm of base $2$.

\subsection{Proof outlines} Theorem \ref{main1} is arguably less surprising than Theorem \ref{main2}. It is already known that the analyticity of $K$ is not required for the $\gamma_{n,m}$ almost sure lower bound $\overline{\dim}_{\mathrm{B}} \pi_{V}(K) \geq \min\{\dim_{\mathrm{H}} K,m\}$, where $\overline{\dim}_{\mathrm{B}}$ stands for upper box dimension. For a short proof in $\R^{2}$ using combinatorial arguments, see \cite[Theorem 4.3]{MR3345668}. One would, then, like to reduce the proof of Theorem \ref{main1} to this known case via the formula
\begin{displaymath} \dim_{\mathrm{P}} \pi_{V}(K) = \inf \big\{ \sup_{i} \overline{\dim}_{\mathrm{B}\,} F_{i} : \pi_{V}(K) \subset \bigcup_{i} F_{i} \big\}. \end{displaymath}
The only issue is that the most obvious reduction uses a Frostman measure supported on $K$, as in \cite[Lemma 4.5]{MR3345668}, and this approach is not available for analytic sets. In the end, however, it turns out Frostman measures can be dispensed with by additional combinatorial arguments.

Let us then discuss Theorem \ref{main2}. Assume for simplicity that $0 \leq s \leq m$ and $\mathcal{H}^{s}(K) > 0$, where $s := \dim_{\mathrm{H}} K = \dim_{\mathrm{P}} K$. Ignoring some technicalities, the equality of Hausdorff and packing dimension implies that for all small $\delta > 0$, the set $K$ can be covered by a family of $\approx \delta^{-s}$ balls of radius $\delta$ whose centres $K_{\delta}$ form a $\delta$-discretised $s$-dimensional set (Definition \ref{CdeltaSet}). By an argument essentially due to Marstrand, the set $K_{\delta}$ has the following property: there is a "tiny" exceptional set $G_{\delta,\mathrm{bad}} \subset G(n,m)$ such that $N(\pi_{V}(K_{V,\delta}),\delta) \approx \delta^{-s}$ for all $V \in G(n,m) \, \setminus \, G_{\delta,\mathrm{bad}}$ and for all subsets $K_{V,\delta} \subset K_{\delta}$ with $|K_{V}| \approx \delta^{-s}$. For a more precise statement, see Lemma \ref{lemma2}.

I can now sketch the proof of Theorem \ref{main2}. Assume that the conclusion fails: there exists a set $G \subset G(n,m)$ with $\gamma_{n,m}(G) > 0$ such that $\dim_{\mathrm{H}} \pi_{V}(K) < \dim_{\mathrm{H}} K$ for all $V \in G$. By a pigeonholing argument, this implies that there exists a scale $\delta > 0$ and "non-tiny" subset $G_{\delta} \subset G$ such that the following holds for all $V \in G_{\delta}$: there is a set $K_{V} \subset K$ with $\mathcal{H}^{s}_{\infty}(K_{V}) > 0$ such that $N(\pi_{V}(K_{V}),\delta) \ll \delta^{-s}$. Since $K_{V} \subset K$ is contained in the $\delta$-neighbourhood of $K_{\delta}$, and $\mathcal{H}^{s}_{\infty}(K_{V}) > 0$, the set $K_{V,\delta} := \{p \in K_{\delta} : \dist(p,K_{V}) \leq \delta\}$ satisfies $|K_{V,\delta}| \approx \delta^{-s}$, and so the projection theorem stated above applies: since $G_{\delta}$ was "non-tiny", there exists a plane $V \in G_{\delta} \, \setminus \, G_{\delta,\mathrm{bad}}$, hence $N(\pi_{V}(K_{V}),\delta) \sim N(\pi_{V}(K_{V,\delta}),\delta) \approx \delta^{-s}$. This contradicts the choice of $K_{V}$ and completes the proof of Theorem \ref{main2}.

\section{Discretising fractals}\label{s:prelim}
Recall that the \emph{Hausdorff dimension} of a set $K \subset \R^{n}$ is the number $\Hd K = \inf \{s \geq 0 : \calH^{s}_{\infty}(K) = 0\}$. Here $\calH^{s}_{\infty}(K)$ is the \emph{$s$-dimensional Hausdorff content}
\begin{displaymath} \calH^{s}_{\infty}(K) = \inf \left\{ \sum_{i} \diam(U_{i})^{s} : K \subset \bigcup_{i} U_{i} \right\}. \end{displaymath}
I next recall \emph{packing dimension}; for more information, see \cite[\S 5.9]{zbMATH01249699} or \cite[\S 3.4]{MR3236784}. 
\begin{definition}[Packing and upper box dimensions]\label{d:pd} The \emph{packing dimension} of a set $K \subset \R^{n}$ is the number
\begin{displaymath} \Pd K = \inf \left\{ \sup_{i} \Bd F_{i} : F_{i} \text{ is bounded and } K \subset \bigcup_{i \in \N} F_{i} \right\}. \end{displaymath} 
Here $\Bd$ is the \emph{upper box dimension}, defined for bounded sets $F \subset \R^{n}$ by
\begin{displaymath} \Bd F := \limsup_{\delta \to 0} \frac{\log N(F,\delta)}{-\log \delta}, \end{displaymath}
where $N(F,\delta)$ stands for the least number of $\delta$-balls required to cover $F$. 
\end{definition}

It may be worth pointing out that packing dimension can be defined in two alternative ways (via upper box dimension, as above, or via \emph{packing measures}), but the two notions coincide for arbitrary sets in $\R^{n}$, see \cite[Theorem 5.11]{zbMATH01249699}. I also remark here that the variant of Theorem \ref{main1} for $\Bd$ (in place of $\Pd$) is remarkably simple; this can be inferred by combining \cite[Proposition 4.10]{MR3345668} and \cite[Proposition A.1]{MR3254928}. 

The next definitions are, to the best of my knowledge, due to Katz and Tao \cite{MR1856956}:

\begin{definition}[$(C,\delta,s)$-sets]\label{CdeltaSet} Let $\delta,s > 0$ and $C \geq 1$. A finite set $P \subset \R^{n}$ is called a \emph{$(C,\delta,s)$-set} if
\begin{displaymath} |P \cap B(x,r)| \leq C\left(\frac{r}{\delta} \right)^{s}, \qquad x \in \R^{n}, \, \delta \leq r \leq 1. \end{displaymath}
\end{definition}
Here $|\cdot|$ refers to cardinality. I will informally talk of "$(\delta,s)$-sets" if the constant $C$ is not important. A good, if imprecise, heuristic is that a $(\delta,s)$-set looks like a $\delta$-net inside a set of Hausdorff dimension $s$. There are various ways of making this more precise. For example, \cite[Proposition A.1]{MR3254928} shows that any set $K \subset \R^{3}$ with $\mathcal{H}^{s}_{\infty}(K) =: \tau > 0$ contains a $(\delta,s)$-set of cardinality $\gtrsim \tau \cdot \delta^{-s}$. A "converse" way to relate Hausdorff dimension and $(\delta,s)$-sets is Lemma \ref{lemma1} below, due to Katz and Tao \cite{MR1856956}, which states that an arbitrary subset of $\R^{n}$ with $\Hd K < s$ can be \emph{strongly covered} by $\delta$-neighbourhoods of $(\delta,s)$-sets.

\begin{definition}[Strong covering] A sequence of sets $E_{1},E_{2},\ldots$ \emph{strongly covers} another set $F$ if every point of $F$ is contained in infinitely many of the sets $E_{i}$. \end{definition} 

The next lemma, and its proof, are virtually the same as \cite[Lemma 7.5]{MR1856956}. I include all the details, because \cite[Lemma 7.5]{MR1856956} is only stated for compact sets, and it also uses the terminology of "hyper-dyadic rationals" which I prefer to avoid here.

\begin{lemma}\label{lemma1} Let $0 < s \leq n$, and let $K \subset \R^{n}$ be a set with $\Hd K < s$. Then there exists a constant $C \geq 1$, depending only on $n,s$, and $\Hd K$ such that the following holds. For every $k \in 2^{-\N}$ there exists a $(Ck^{2},2^{-k},s)$-set $P_{k}$ such that the sequence $\{P_{k}(C_{n}2^{-k})\}_{k \in \N}$ strongly covers $K$. Here $C_{n} \geq 1$ only depend on $n$. \end{lemma}
Here, and in the rest of the paper, $A(\delta)$ refers to the $\delta$-neighbourhood of $A \subset \R^{n}$.

\begin{proof}[Proof of Lemma \ref{lemma1}] Fix $\epsilon > 0$ such that $\Hd K < s - \epsilon$. Then, for every $i \in \{1,2,\ldots\}$, find a collection $\mathcal{Q}_{i}$ of disjoint dyadic cubes of side-length \textbf{at most $2^{-i}$} which cover $K$ and satisfy
\begin{equation}\label{form22} \sum_{Q \in \mathcal{Q}_{i}} \ell(Q)^{s - \epsilon} \leq 1. \end{equation}
Above $\ell(Q)$ refers to the side-length of $Q$. For $i \geq j$, write further $\mathcal{Q}_{i,j} := \{Q \in \mathcal{Q}_{i} : \ell(Q) = 2^{-j}\}$, and let $X_{i,j}$ be the union of the cubes in $\mathcal{Q}_{i,j}$. Now, picking the centres of the cubes in $\mathcal{Q}_{i,j}$, we could obtain a $2^{-j}$-separated set $P_{i,j}$ of cardinality $|P_{i,j}| \lesssim 2^{js}$, whose $2^{-j}$-neighbourhood is essentially $X_{i,j}$. Then, a natural first attempt at $P_{j}$, $j \in \N$, would be the union $P_{j} := \cup \{P_{i,j} : 1 \leq i \leq j\}$, with $|P_{j}| \lesssim j \cdot 2^{js}$, whose $C_{n}2^{-j}$-neighbourhood contains
\begin{displaymath} \cup \{X_{i,k} : 1 \leq i \leq k\} =: X_{j}. \end{displaymath}
Note that the sets $X_{j}$ strongly cover $K$. The problem is, however, that $P_{j}$ is not necessarily a $(Cj^{2},2^{-j},s)$-set, so a further refinement is needed.

For $1 \leq i \leq j$ fixed, choose another collection $\mathcal{Q}_{i,j}'$ of disjoint dyadic cubes of side-length \textbf{at least $2^{-j}$} which 
\begin{enumerate}
\item covers $X_{i,j}$, and
\item minimises the sum $\sum \{\ell(Q)^{s} : Q \in \mathcal{Q}_{i,j}'\}$ among all (disjoint) dyadic covers of $X_{i,j}$. 
\end{enumerate}
It is easy to see that a minimiser exists, since cubes of side-length exceeding $2^{-\epsilon j/s}$ need not be considered. Indeed, if $\mathcal{Q}_{i,j}'$ contained a cube of such side-length, then also the sum in (2) would exceed $2^{-\epsilon j}$. However, the collection $\mathcal{Q}_{i,j}$ is a cover for $X_{i,j}$, and satisfies
\begin{displaymath} \sum_{Q \in \mathcal{Q}_{i,j}} \ell(Q)^{s} \leq 2^{-\epsilon j} \sum_{Q \in \mathcal{Q}_{i,j}} \ell(Q)^{s - \epsilon} \stackrel{\eqref{form22}}{\leq} 2^{-\epsilon j}, \end{displaymath}
using that $\ell(Q) \leq 2^{-j}$ for all $Q \in \mathcal{Q}_{i,j}$. So, we know that $\mathcal{Q}_{i,j}'$ is a collection of dyadic cubes of side-lengths between $2^{-j}$ and $2^{-\epsilon j/s}$. Moreover, if $Q_{0} \subset \R^{n}$ is an arbitrary dyadic cube, then
\begin{equation}\label{form21} \sum \left\{ \ell(Q)^{s} : Q \in  \mathcal{Q}_{i,j}' \text{ and } Q \subset Q_{0}\right\} \leq \ell(Q_{0})^{s}, \end{equation}
since otherwise the sum in (2) could be further reduced. 

Next, for $k \in \N$, let
\begin{displaymath} \mathcal{Q}_{i,j,k}' := \{Q \in \mathcal{Q}_{i,j}' : \ell(Q) = 2^{-k}\}, \end{displaymath}
so that $\mathcal{Q}_{i,j}' \subset \cup_{k} \mathcal{Q}_{i,j,k}'$. We record that 
\begin{equation}\label{form23} \mathcal{Q}_{i,j,k}' = \emptyset, \qquad k \notin \{\floor{\epsilon j/s},\ldots,j\}. \end{equation}
Let $P_{i,j,k}$ be the collection of the centres of the cubes in $\mathcal{Q}_{i,j,k}'$. It follows from \eqref{form21} that $P_{i,j,k}$ is a $(C,2^{-k},s)$-set for some $C = C(n) \geq 1$. For $k \in \{1,2,\ldots\}$ fixed, we define
\begin{displaymath} P_{k} := \bigcup_{i = 1}^{\infty} \bigcup_{j \geq i} P_{i,j,k}. \end{displaymath}
We claim that $K$ is strongly covered by the sequence $\{P_{k}(C2^{-k})\}_{k \in \N}$, and that $P_{k}$ is a $(Ck^{2},2^{-k},s)$-set with $C \sim_{n} (s/\epsilon)^{2}$. To see the first property, fix $x \in K$. Then, for every $i \in \N$, $x$ is contained in $X_{i,j(i)} \subset \cup \mathcal{Q}_{i,j(i)}'$ for some $j(i) \geq i$. Consequently, 
\begin{displaymath} x \in \cup \mathcal{Q}_{i,j(i),k(i,j)}' \subset P_{i,j(i),k(i,j)}(C_{n}2^{-k(i,j)}) \end{displaymath}
for some $k(i,j) \geq \epsilon j(i)/s \geq \epsilon i/s$, using also \eqref{form23}. This implies that $x \in P_{k(i,j)}(C_{n}2^{-k(i,j)})$, and since $k(i,j) \to \infty$ as $j \to \infty$, we conclude the strong covering property.

To verify the $(\delta,s)$-set property, note that if $j > \ceil{ks/\epsilon}$ then $k < \floor{\epsilon j/s}$, hence $P_{i,j,k} = \emptyset$ by \eqref{form23}. It follows that we may re-write
\begin{displaymath} P_{k} = \bigcup_{i = 1}^{\ceil{ks/\epsilon}} \bigcup_{j = i}^{\ceil{ks/\epsilon}} P_{i,j,k}. \end{displaymath}
Since each $P_{i,j,k}$ was individually a $(C,2^{-k},s)$-set, it follows that $P_{k}$ is a $(C(ks/\epsilon)^{2},2^{-k},s)$-set, as claimed. \end{proof}

The next lemma concerns the orthogonal projections of $(\delta,s)$-sets. It is a $\delta$-discretised version of the following result of Marstrand \cite{MR0063439}: if $K \subset \R^{n}$ is a compact set with $0 < \calH^{s}(K) < \infty$, then there exists a set $G_{\mathrm{bad}} \subset G(n,m)$ of zero $\gamma_{n,m}$ measure such that the following holds. Whenever $V \in G(n,m) \, \setminus \, G_{\mathrm{bad}}$, and $K' \subset K$ satisfies $\calH^{s}(K') > 0$, then $\Hd \pi_{V}(K') = \min\{m,s\}$. In particular, the "exceptional" set $G_{\mathrm{bad}}$ is independent $K'$.

\begin{lemma}\label{lemma2} Let $0 < m < n$, $0 \leq s \leq n$, and $\epsilon > 0$. Then the following holds for all $0 < \delta < \delta_{0}$, where $\delta_{0}$ depends only on $\epsilon,n,s$ Let $P \subset B(0,1) \subset \R^{n}$ be a $(\delta^{-\epsilon},\delta,s)$-set. Then, there exists a set $G_{\mathrm{bad}} \subset G(n,m)$ with $\gamma_{n,m}(G_{\mathrm{bad}}) \leq \delta^{\epsilon}$, and the following property. If $V \in G(n,m) \, \setminus \, G_{\mathrm{bad}}$, and if $P' \subset P$ with $|P'| \geq \delta^{-s + \epsilon}$, then
\begin{equation}\label{form19} N(\pi_{V}(P'),\delta) \geq \delta^{-\min\{s,m\} + 6\epsilon}. \end{equation} 
\end{lemma}
\begin{proof} Let $P \subset \R^{n}$ be a $(\delta^{-\epsilon},\delta,s)$-set, as in the hypothesis. One may easily check that
\begin{equation}\label{form20} \sum_{x \in P} \mathop{\sum_{y \in P}}_{x \neq y} \frac{1}{|x - y|^{m}} \lesssim \begin{cases} \delta^{-2s - 2\epsilon}, & \text{if } m < s \leq n, \\ \delta^{-m - s - 3\epsilon}, & \text{if } 0 \leq s \leq m. \end{cases} \end{equation}
Indeed, just divide the inner summation into dyadic annuli and use the $(\delta^{-\epsilon},\delta,s)$-set condition, and finally also observe that $|P| = |P \cap B(0,1)| \leq \delta^{-\epsilon - s}$. For $V \in G(n,m)$ fixed, define next the quantity
\begin{displaymath} \mathcal{E}_{V}(P) := |\{(x,y) \in P \times P : |\pi_{V}(x) - \pi_{V}(y)| \leq \delta\}|, \end{displaymath} 
and note immediately that
\begin{equation}\label{form26} \mathcal{E}_{V}(P) = |\{(x,y) \in P \times P : x \neq y \text{ and } |\pi_{V}(x) - \pi_{V}(y)| \leq \delta\}| + |P| =: \mathcal{E}'_{V}(P) + |P|. \end{equation}
We recall from \cite[Lemma 3.11]{zbMATH01249699} the following geometric estimate:
\begin{equation}\label{form25} \gamma_{n,m}(\{V \in G(n,m) : |\pi_{V}(x) - \pi_{V}(y)| \leq \delta\}) \lesssim_{n} \frac{\delta^{m}}{|x - y|^{m}}, \qquad x \neq y. \end{equation} 
Combining \eqref{form20},\eqref{form26}, and \eqref{form25},  and noting that $|P| \leq \delta^{-s - \epsilon}$, we find that
\begin{equation}\label{form18} \int_{G(n,m)} \mathcal{E}_{V}(P) \, d\gamma_{n,m}(V) \lesssim_{n} \mathop{\sum_{x,y \in P}}_{x \neq y} \frac{\delta^{m}}{|x - y|^{m}} + |P| \lesssim \begin{cases} \delta^{m - 2s - 2\epsilon}, & \text{if } m < s \leq n, \\ \delta^{-s -3\epsilon}, & \text{if } 0 \leq s \leq m. \end{cases}. \end{equation} 
Now, write $G_{\mathrm{bad}} \subset G(n,m)$ for the set of planes $V \in G(n,m)$ such that \eqref{form19} fails for some set $P' \subset P$ with $|P'| \geq \delta^{-s + \epsilon}$. We claim that 
\begin{displaymath} G_{\mathrm{bad}} \subset \{V \in G(n,m) : \mathcal{E}_{V}(P) \gtrsim \delta^{\min\{s,m\} - 2s - 4\epsilon}\}. \end{displaymath}
Indeed, divide $V \in G_{\mathrm{bad}}$ into dyadic cubes $Q$ with $\diam(Q) \in [\tfrac{\delta}{2},\delta]$, and let $\calD$ be the family of these cubes meeting $\pi_{V}(P')$. Then $|\calD| \lesssim \delta^{-\min\{s,m\} + 6\epsilon}$ by hypothesis. Moreover,
\begin{align*} \mathcal{E}_{V}(P) & \geq \sum_{Q \in \calD} |\{(x,y) \in P' \times P' : \pi_{V}(x),\pi_{V}(y) \in Q\}|\\
& = \sum_{Q \in \calD} |P' \cap \pi_{V}^{-1}(Q)|^{2} \geq \frac{1}{|\calD|} \left( \sum_{Q \in \calD} |P' \cap \pi_{V}^{-1}(Q)| \right)^{2} \gtrsim \delta^{\min\{s,m\} - 2s - 4\epsilon}, \end{align*} 
using Cauchy-Schwarz once. Finally, by Chebyshev's inequality and \eqref{form18}, 
\begin{displaymath} \gamma_{n,m}(G_{\mathrm{bad}}) \leq \gamma_{n,m}(\{V \in G(n,m) : \mathcal{E}_{V}(P) \gtrsim \delta^{\min\{s,m\} - 2s - 4\epsilon}\}) \leq \delta^{\epsilon}, \end{displaymath}
at least for all $\delta > 0$ small enough. This completes the proof. \end{proof}

The final lemma clarifies the connection between a set $K \subset \R^{n}$ with $N(K,\delta) \approx \delta^{-s}$, and $(\delta,s)$-sets: the part of $K$ which is not contained in the $\delta$-neighourhood of a single $(\delta,s)$-set has small $s$-dimensional Hausdorff content. This "bad" part can easily be all of $K$, however: the lemma is only useful if we have an \emph{a priori} lower bound on $\calH^{s}_{\infty}(K)$.

\begin{lemma}\label{lemma3} Let $0 \leq s \leq n$, $\delta > 0$, $C \geq 1$, and let $K \subset \R^{n}$ be a bounded set with
\begin{equation}\label{form9} N(K,\delta) \leq C\delta^{-s}. \end{equation} 
Then, for any $L \geq 1$, there exists a disjoint decomposition $K = K_{\mathrm{good}} \cup K_{\mathrm{bad}}$ such that
\begin{enumerate}
\item $\calH^{s}_{\infty}(K_{\mathrm{bad}}) \lesssim L^{-1}$, and 
\item $K_{\mathrm{good}}$ is contained in the $\delta$-neighbourhood of a $(CL,\delta,s)$-set.
\end{enumerate}
The implicit constant in \textup{(1)} only depends on $n$.
\end{lemma}

\begin{proof} Assume with no loss of generality that $\delta \in 2^{-\Z}$. Let $\calD_{\delta}$ be the collection of dyadic cubes $Q \subset \R^{n}$ of side-length $\ell(Q) = \delta$, and let $\calD_{\geq \delta}$ be the collection of dyadic cubes $Q \subset \R^{n}$ of side-length $\ell(Q) \geq \delta$. Finally, let
\begin{displaymath} \calD_{\delta}(K) := \{Q_{\delta} \in \calD_{\delta} : Q_{\delta} \cap K \neq \emptyset\}, \end{displaymath}
so $|\calD_{\delta}(K)| \lesssim C\delta^{-s}$ by \eqref{form9}. A cube $Q \in \calD_{\geq \delta}$ is called \emph{heavy} if
\begin{displaymath} |\{Q_{\delta} \in \calD_{\delta}(K) : Q_{\delta} \subset Q\}| \geq \tau CL\left(\frac{\ell(Q)}{\delta} \right)^{s}. \end{displaymath} 
Here $\tau = \tau(n) > 0$ is a small constant to be specified later. Note that arbitrarily large cubes cannot be heavy by the upper bound on $|\mathcal{D}_{\delta}(K)|$. Let $K_{\mathrm{bad}} \subset K$ be the set of points in $K$ which are contained in at least one heavy cube. Then $K_{\mathrm{bad}}$ is covered by the maximal heavy cubes in $\calD_{\geq \delta}$, denoted $\mathcal{M}$. The cubes in $\mathcal{M}$ are disjoint, and moreover
\begin{displaymath} \calH^{s}_{\infty}(K_{\mathrm{bad}}) \lesssim \sum_{Q \in \mathcal{M}} \ell(Q)^{s} \leq \frac{\delta^{s}}{\tau CL} \sum_{Q \in \mathcal{M}} |\{Q_{\delta} \in \calD_{\delta}(K) : Q_{\delta} \subset Q\}| \leq \frac{1}{\tau L}. \end{displaymath} 
This verifies condition (1). Define $K_{\mathrm{good}} := K \, \setminus \, K_{\mathrm{bad}}$. By definition, no point in $K_{\mathrm{good}}$ is contained in a heavy cube. In other words, if $Q \in \mathcal{D}_{\geq \delta}$ is arbitrary, then either $Q \cap K_{\mathrm{good}} = \emptyset$, hence $N(K_{\mathrm{good}} \cap Q,\delta) = 0$, or alternatively $Q$ is not a heavy cube. In this case
\begin{displaymath} N(K_{\mathrm{good}} \cap Q,\delta) \leq N(K \cap Q,\delta) \lesssim_{n} |\{Q_{\delta} \in \calD_{\delta}(K) : Q_{\delta} \subset Q\}| < \tau CL\left(\frac{\ell(Q)}{\delta} \right)^{s}. \end{displaymath} 
These estimates imply that if $P \subset K_{\mathrm{good}}$ is a maximal $\delta$-separated subset, then $|P \cap B(x,r)| \lesssim_{n} \tau CL(r/\delta)^{s}$ for all $x \in \R^{n}$ and $r \geq \delta$. Thus $P$ is a $(CL,\delta,s)$-set if $\tau = \tau(n) > 0$ is small enough, and of course $K_{\mathrm{good}} \subset P(\delta)$.\end{proof}

\section{Proofs of the main theorems}

Here is the version of the \emph{pigeonhole principle} that will be frequently employed:
\begin{pigeon}\label{pig} Let $\{a_{1},a_{2},\ldots\}$ be a sequence of non-negative numbers, and write $\sum a_{j} =: A$. Then there exists an index $j \in \N$ such that $a_{j} \gtrsim A/j^{2}$. \end{pigeon}
This principle will be typically employed so that we have a set $K$ in an (outer) measure space $(X,\mu)$ with $\mu(K) > 0$, and a cover $U_{1},U_{2},\ldots$ for $K$. Then, by the sub-additivity of $\mu$, and Pigeon \ref{pig}, we may infer that $\mu(U_{j}) \gtrsim \mu(K)/j^{2}$ for some $j \in \N$.

\subsection{Proof of Theorem \ref{main2}} Write $t := \Hd K = \Pd K$. Recall that the aim is to prove
\begin{displaymath} \Hd \pi_{V}(K) = \min\{t,m\} \qquad \text{for $\gamma_{m,n}$ a.e. } V \in G(n,m). \end{displaymath}
To reach a contradiction, assume with no loss of generality that $t > 0$, and there exists $0 < u < \min\{t,m\}$, and a $\gamma_{n,m}$ positive-measure subset $G \subset G(n,m)$ such that 
\begin{equation}\label{form12} \Hd \pi_{V}(K) < u, \qquad V \in G. \end{equation} 
Pick also $0 < s < t$ so close to $t$ that still $\min\{s,m\} > u$. Then $\mathcal{H}^{s}_{\infty}(K) > 0$. Fix $\epsilon > 0$. By definition of $\dim_{\mathrm{P}} K = t$ and the countable sub-additivity of Hausdorff content, there exists a bounded subset $K_{\epsilon} \subset K$ which satisfies both $\mathcal{H}^{s}_{\infty}(K_{\epsilon}) > 0$ and $\Bd K_{\epsilon} \leq t + \tfrac{\epsilon}{2}$. We replace $K$ by $K_{\epsilon}$ without changing notation. Then, the following holds for all $\delta > 0$ small enough, depending on $\epsilon$ and $K = K_{\epsilon}$:
\begin{equation}\label{form11} N(K,\delta) \le \delta^{-t - \epsilon} = [\delta^{s - t - \epsilon}] \cdot \delta^{-s}. \end{equation}
To fix the parameters, we will eventually need to pick $s < t$ so close to $t$, and $\epsilon > 0$ so small, that
\begin{equation}\label{form17} \min\{s,m\} - 6(t + 2\epsilon - s) > u. \end{equation}
By \eqref{form12}, for every $V \in G$ and $\delta_{0} > 0$, there exists a collection of dyadic cubes $\mathcal{Q}_{V}$ on $V$ of side-lengths $\leq \delta_{0}$ with the properties
\begin{equation}\label{form13} \pi_{V}(K) \subset \bigcup_{Q \in \mathcal{Q}_{V}} Q \quad \text{and} \quad \sum_{Q \in \mathcal{Q}_{V}} \ell(Q)^{u} \leq 1. \end{equation}
We will eventually need to choose $\delta_{0} > 0$ small in a way which depends on $\epsilon,s,t,u$, and $n$. The cubes in $\mathcal{Q}_{V}$ are dyadic, so they can be further partitioned into collections $\mathcal{Q}_{V}(j)$ of (disjoint) cubes of side-length $2^{-j} \leq \delta_{0}$. Write $K_{V}^{j} := \{x \in K : \pi_{V}(x) \in \cup \, \mathcal{Q}_{V}(j)\}$ for the the part of $K$ whose $\pi_{V}$-projection is covered by the intervals in $\mathcal{Q}_{V}^{j}$. In particular,
\begin{equation}\label{form14} N(\pi_{V}(K_{V}^{j}),2^{-j}) \lesssim 2^{ju} \end{equation} 
by \eqref{form13}. Since $\pi_{V}(K) \subset \cup \, \mathcal{Q}_{V}$ for $V \in G$, and $\calH^{s}_{\infty}(K) > 0$, we may use Pigeon \ref{pig} to find a dyadic scale $2^{-j(V)} \leq \delta_{0}$ such that 
\begin{equation}\label{form10} \calH^{s}_{\infty}(K_{V}^{j(V)}) \gtrsim_{K,s} j(V)^{-2}. \end{equation}
The choice of the index $j(V) \in \N$ \emph{a priori} depends on $V \in G$, but we may practically eliminate this dependence by another appeal to Pigeon \ref{pig}. Let $G_{j} := \{V \in G : j(V) = j\}$. Then, the sets $G_{j}$, $2^{-j} \leq \delta_{0}$, cover the $\gamma_{n,m}$ positive-measure set $G$, so there exists a fixed index $j \in \N$ such that $\gamma_{n,m}(G_{j}) \gtrsim j^{-2}$. We then record that \eqref{form14}-\eqref{form10} hold for every $V \in S_{j}$ with $j(V) = j$. To simplify notation, write $K_{V} := K_{V}^{j}$ for $V \in G_{j}$.

Next, we write $\delta := 2^{-j} \leq \delta_{0}$, and apply Lemma \ref{lemma3} with constant $C := \delta^{s - t - \epsilon}$, and level $L := \delta^{-\epsilon}$. By \eqref{form11}, the main hypothesis \eqref{form9} of Lemma \ref{lemma3} is satisfied. The conclusion is that $K = K_{\mathrm{good}} \cup K_{\mathrm{bad}}$ with the properties that
\begin{displaymath} \calH^{s}_{\infty}(K_{\mathrm{bad}}) \lesssim_{n} \delta^{\epsilon}, \end{displaymath} 
and $K_{\mathrm{good}}$ is contained in the $\delta$-neighbourhood of a single $(\delta^{s - t - 2\epsilon},\delta,s)$-set $P \subset \R^{n}$. Since $K_{\mathrm{good}} \subset K$ is bounded, there is no loss of generality assuming that $P \subset B(0,1)$. We write $\epsilon' := t + 2\epsilon - s$, and apply Lemma \ref{lemma2} to the set $P$ and the parameter $\epsilon'$: there exists a subset $G_{\mathrm{bad}} \subset G(n,m)$ with $\gamma_{n,m}(G_{\mathrm{bad}}) \leq \delta^{\epsilon'}$ such that whenever $V \in G(n,m) \, \setminus \, G_{\mathrm{bad}}$ and $P' \subset P$ has cardinality $|P'| \geq \delta^{-s + \epsilon'}$, we have
\begin{equation}\label{form24} N(\pi_{V}(P'),\delta) \geq \delta^{-\min\{s,m\} + 6\epsilon'}. \end{equation}
Fix $V \in G$, and recall from \eqref{form10} that 
\begin{displaymath} \calH^{s}_{\infty}(K_{V}) \gtrsim j^{-2} = \log^{-2} \tfrac{1}{\delta}. \end{displaymath}
Since $K_{V} \subset K \subset K_{\mathrm{good}} \cup K_{\mathrm{bad}}$, and $\calH^{s}_{\infty}(K_{\mathrm{bad}}) \lesssim \delta^{\epsilon} \ll \log^{-2} \tfrac{1}{\delta}$ (take the parameter $\delta_{0} \geq \delta$ so small that this works), we infer from the sub-additivity of $\calH^{s}_{\infty}$ that
\begin{displaymath} \calH^{s}_{\infty}(K_{V} \cap P(\delta)) \geq \calH^{s}_{\infty}(K_{V} \cap K_{\mathrm{good}}) \gtrsim \log^{-2} \tfrac{1}{\delta}, \qquad V \in G. \end{displaymath}
It follows that there exists a set $P_{V} \subset P \cap K_{V}(\delta)$ of cardinality $|P_{V}|\gtrsim \delta^{-s} \log^{-2} \tfrac{1}{\delta}$. In particular, if $\delta > 0$ is small enough (which can be arranged by taking $\delta_{0} \geq \delta$ small enough to begin with), we have $|P_{V}| \geq \delta^{-s + \epsilon'}$. Now, \eqref{form24} implies that 
\begin{equation}\label{form16} N(\pi_{V}(K_{V}),\delta) \gtrsim N(\pi_{V}(P_{V}),\delta) \geq \delta^{-\min\{s,m\} + 6\epsilon'} \end{equation}
for all $V \in G(n,m) \, \setminus \, G_{\mathrm{bad}}$. But $\gamma_{n,m}(G_{\mathrm{bad}}) \leq \delta^{\epsilon'} \ll \log^{-2} \tfrac{1}{\delta} \lesssim \gamma_{n,m}(G_{j})$, so we infer that \eqref{form16} holds for some $V \in G_{j}$. Recalling the choice of $\epsilon' = t + 2\epsilon - s$, and in particular that $\min\{s,m\} - 6\epsilon' > u$ by \eqref{form17}, we find that \eqref{form16} contradicts \eqref{form14} for any $V \in G_{j}$ (again: for $\delta > 0$ small enough). This contradiction completes the proof of Theorem \ref{main2}.

\subsection{Proof of Theorem \ref{main1}} Write $s : = \Hd K \in [0,n]$, and recall that the aim is to prove
\begin{equation}\label{form15} \Pd \pi_{V}(K) \geq \min\{s,m\} \qquad \text{for $\gamma_{n,m}$ a.e. } V \in G(n,m). \end{equation}
If $s = 0$, this is clear, so we may assume that $s > 0$. We may also assume, using the countable stability of $\Hd$, that $K$ is bounded, and then that $K \subset B(0,1)$. Pick $s' < s$, so
\begin{equation}\label{form1} \calH^{s'}_{\infty}(K) > 0. \end{equation}
Pick also $s'' > s$, and write
\begin{equation}\label{form4}  \epsilon := 2(s'' - s') > s'' - s'. \end{equation}
Using Lemma \ref{lemma1}, pick a sequence of $(C_{\delta},\delta,s'')$-sets $\{P_{\delta}\}$, for $\delta \in 2^{-\N}$, such that $K$ is strongly covered the $C_{n}\delta$-neighbourhoods $P_{\delta}(C\delta)$. Here
\begin{displaymath} C_{\delta} \lesssim_{n,s,s''} \log^{2} \tfrac{1}{\delta}. \end{displaymath} 
In particular, $C_{\delta} \leq \delta^{-\epsilon}$ for all $0 < \delta \leq \delta_{0}$, where $\delta_{0} > 0$ only depends on $n,s,s'$, and $s''$. For every $\delta \in 2^{-\N}$ with $\delta < \delta_{0}$, and for $\epsilon = 2(s'' - s')$ as in \eqref{form4}, let $G^{\delta}_{\mathrm{bad}} \subset G(n,m)$ be the exceptional set given by Lemma \ref{lemma2} (associated to $P_{\delta}$) with $\gamma_{n,m}(G^{\delta}_{\mathrm{bad}}) \leq \delta^{\epsilon}$, and such that
\begin{equation}\label{form6} N(\pi_{V}(P'),\delta) \geq \delta^{-\min\{s'',m\} + 6\epsilon}, \qquad V \in G(n,m) \, \setminus \, G^{\delta}_{\mathrm{bad}}, \end{equation} 
whenever $P' \subset P_{\delta}$ satisfies $|P'| \geq \delta^{-s'' + \epsilon}$. By the Borel-Cantelli lemma, the set
\begin{displaymath} G_{\mathrm{bad}} := \{V \in G(n,m) : V \in G_{\mathrm{bad}}^{\delta} \text{ for infinitely many $\delta \in 2^{-\N}$}\} \end{displaymath} 
has $\gamma_{n,m}(G_{\mathrm{bad}}) = 0$. Pick $V \in G(n,m) \, \setminus \, G_{\mathrm{bad}}$. We claim that
\begin{equation}\label{form7} \Pd \pi_{V}(K) \geq \min\{s'',m\} - 6\epsilon, \end{equation} 
which is evidently good enough to prove \eqref{form15} (by letting $s',s'' \to s$, and recalling that $\epsilon = 2(s'' - s')$). If \eqref{form7} fails, then, by definition of $\Pd$, there exists a number $t < \min\{s'',m\}$, and bounded sets $F_{V,1},F_{V,2},\ldots \subset V$ such that
\begin{displaymath} \pi_{V}(K) \subset \bigcup_{i \in \N} F_{V,i}, \end{displaymath}
with the property that
\begin{displaymath} \Bd F_{V,i} < t - 6\epsilon, \qquad i \in \N. \end{displaymath}
In particular, by \eqref{form1} and the sub-additivity of Hausdorff content, there exists a subset $K_{V} \subset K$ with $\calH_{\infty}^{s'}(K_{V}) > 0$, and an index $i \in \N$, with the property that $\pi_{e}(K_{V}) \subset F_{V,i}$. As a consequence,
\begin{equation}\label{form5} \limsup_{\delta \to 0} \frac{\log N(\pi_{V}(K_{V}),\delta)}{-\log \delta} = \Bd \pi_{V}(K_{V}) \leq \Bd \pi_{V}(F_{V,i}) < t - 6\epsilon. \end{equation} 
The sets $P_{\delta}(C_{n}\delta)$ strongly cover $K$, so they also strongly cover $K_{V}$. In other words, every point in $K_{V}$ is contained in infinitely many sets in $P_{\delta}(C_{n}\delta)$. It follows from the "easier" Borel-Cantelli lemma (which only requires the sub-additivity of the Hausdorff content $\calH^{s'}_{\infty}$) that 
\begin{equation}\label{form2}  \mathop{\sum_{\delta \in 2^{-\N}}}_{\delta < \delta_{0}} \calH^{s'}_{\infty}(P_{\delta}(C_{n}\delta) \cap K_{V}) = \infty \end{equation} 
for any $\delta_{0} \in 2^{-\N}$. Eventually, we will need to pick $\delta_{0} > 0$ small in a way depending on $s',s'',t,n$, and the choice of $V$. Now, we may infer from \eqref{form2} that there exists $\delta < \delta_{0}$ such that
\begin{equation}\label{form3} \calH^{s'}_{\infty}(P_{\delta}(C_{n}\delta) \cap K_{V}) \gtrsim \log^{-2} \tfrac{1}{\delta}. \end{equation} 
It is worth mentioning that we cannot arrange for \eqref{form3} to hold for all $\delta > 0$, but we can have it for arbitrarily small $\delta > 0$, which is good enough for our purposes. There will be a few conditions on how small we want to take $\delta > 0$ (hence $\delta_{0}$). The first one is that $\delta > 0$ should be so small that $V \notin G_{\mathrm{bad}}^{\delta}$; by the initial choice $e \notin G_{\mathrm{bad}}$, this is indeed true for all $\delta > 0$ small enough. A second condition is that $\delta$ should be taken so small that
\begin{equation}\label{form8} N(\pi_{V}(K_{V}),\delta) \leq \delta^{-t + 6\epsilon}. \end{equation}
This is also true by \eqref{form5} for all $\delta > 0$ sufficiently small. After these preliminaries and comments, we infer from \eqref{form3} that there exists a set $P_{V} \subset P_{\delta} \cap K_{V}(C_{n}\delta)$ of cardinality 
\begin{displaymath} |P_{V}| \gtrsim \delta^{-s'} \cdot \log^{-2} \tfrac{1}{\delta}. \end{displaymath}
In particular, if $\delta < \delta_{0}$ is small enough, and recalling from \eqref{form4} that $s'' - \epsilon < s'$, we have $|P_{V}| \geq \delta^{-s'' + \epsilon}$. Since $V \notin G_{\mathrm{bad}}^{\delta}$, this means by \eqref{form6} that
\begin{displaymath} N(\pi_{V}(K_{V}),\delta) \gtrsim_{n} N(\pi_{V}(P_{V}),\delta) \geq \delta^{-\min\{s'',m\} + 6\epsilon}. \end{displaymath}   
Recalling that $t < \min\{s'',m\}$, the inequality above contradicts \eqref{form8}, if $\delta > 0$ is small enough, and the proof of \eqref{form7} is complete. As we pointed out after \eqref{form7}, this also concludes the proof of Theorem \ref{main1}.

\bibliographystyle{plain}
\bibliography{references}

\end{document}